\documentclass[12pt]{article}
\usepackage[cm]{fullpage}

\usepackage{amsmath,amscd, float,times,rotating}
\usepackage{pb-diagram}

\usepackage{latexsym}
\usepackage{amsfonts}
\usepackage{amsmath}
\usepackage{amssymb}

\numberwithin{equation}{section}
\newcommand{\qed}{\hfill \ensuremath{\Box}}

\def\XXint#1#2#3{{\setbox0=\hbox{$#1{#2#3}{\int}$}
\vcenter{\hbox{$#2#3$}}\kern-.5\wd0}}

\newcommand{\dbar}{\overline{\partial}}

\newcommand{\ddt}[1]{\frac{\partial #1}{\partial t}}

\newcommand{\ddbar}{\sqrt{-1}\partial\dbar}

\begin{document}
\newcounter{remark}
\newcounter{theor}
\setcounter{remark}{0} \setcounter{theor}{1}
\newtheorem{claim}{Claim}
\newtheorem{theorem}{Theorem}[section]
\newtheorem{proposition}{Proposition}[section]
\newtheorem{lemma}{Lemma}[section]
\newtheorem{definition}{Definition}[section]
\newtheorem{corollary}{Corollary}[section]
\newenvironment{proof}[1][Proof]{\begin{trivlist}
\item[\hskip \labelsep {\bfseries #1}]}{\end{trivlist}}
\newenvironment{remark}[1][Remark]{\addtocounter{remark}{1} \begin{trivlist}
\item[\hskip \labelsep {\bfseries #1
\thesection.\theremark}]}{\end{trivlist}}

\centerline{\bf \Large On the Convergence of a Modified
K\"{a}hler-Ricci Flow}

\bigskip

\begin{center}{\large Yuan Yuan}

\end{center}

\begin{abstract}
{\small We study the convergence of a modified K\"{a}hler-Ricci flow
defined by Zhou Zhang. We show that the flow converges to a singular
metric when the limit class is degenerate. This proves a conjecture
of Zhang.}
\end{abstract}


\section{Introduction}
The Ricci flow was introduced by Richard Hamilton \cite{H} on Riemannian
manifolds to study the deformation of metrics. Its analogue in K\"{a}hler geometry, the K\"{a}her-Ricci flow, has been
intensively studied in the recent years. It turns out to be a
powerful method to study the canonical metrics on K\"ahler
manifolds. (See, for instance, the papers  \cite{C} \cite{CT} \cite{Pe} \cite{PS} \cite{PSSW} \cite{ST1} \cite{ST2} \cite{TZhu} and the references therein.)  In a
recent paper \cite{Z3},  a modified K\"{a}hler-Ricci
flow was defined by Zhang by allowing the cohomology class to vary artificially. We briefly describe it as follows:

Let $X$ be a closed K\"{a}hler manifold of complex dimension $n$
with a K\"ahler metric $\omega_0$, and let $\omega_\infty$ be a
real, smooth, closed $(1, 1)$-form with $[\omega_\infty]^n=1$. Let
$\Omega$ be a smooth volume form on $X$ such that $\int_X
\Omega=1$. Set $\chi=\omega_0-\omega_\infty$,
$\omega_t=\omega_\infty+e^{-t}\chi$. Let $\varphi: [0,\infty)\times X \rightarrow \mathbb{R}$ be a
smooth function such that ${\widetilde\omega}_t=\omega_t+\ddbar \varphi>0$. Consider the following
Monge-Amp\`{e}re flow:

\begin{equation}\label{maf1}
\ddt{}\varphi = \log \frac{(\omega_t + \ddbar \varphi)^n
}{\Omega};~~~~~~~~~\varphi(0, \cdot)= 0.
\end{equation}

Then the evolution for the corresponding K\"{a}hler metric is given by:

\begin{equation}\label{krf1}
\ddt{} \tilde{\omega}_t = - Ric( \tilde{\omega}_t ) + Ric(\Omega) -
e^{-t}\chi;~~~~~~~~~~\tilde{\omega}_t(0, \cdot)= \omega_0.
\end{equation}

As pointed out by Zhang, the motivation is to apply the geometric flow techniques to study the complex Monge-Amp\`{e}re equation:

\begin{equation}\label{MA}
(\omega_\infty + \ddbar \psi)^n = \Omega.
\end{equation}

This equation has already been intensively studied very recently by
using the pluri-potential theory developed by Bedford-Taylor,
Demailly, Ko{\l}odziej et al. When $\Omega$ is a smooth volume form
and $[\omega_\infty]$ is K\"{a}hler, the equation is solved by Yau
in his solution to the celebrated Calabi conjecture by using the
continuity method \cite{Y1}. When $\Omega$ is $L^p$ with respect to another
smooth reference volume form and $[\omega_\infty]$ is K\"{a}hler,
the continuous solution is obtained by Ko{\l}odziej \cite{K}. Later
on, the bounded solution is obtained
in \cite{EGZ} and \cite{Z2} independently, generalizing
Ko{\l}odziej's theorem to the case when $[\omega_\infty]$ is big and semi-postive, and $\Omega$ is
also $L^p$. On the other
hand, as an interesting question, equation (\ref{MA}) is also
studied on the symplectic manifolds by Weinkove \cite{We}.

In the case of the unnormalized K\"{a}hler-Ricci flow, the evolution
for the cohomology class of the metrics is in the direction of the
canonical class of the manifold. While in the case of (\ref{krf1}),
one can try to deform any initial metric class to an arbitrary
desirable limit class. In particular, on Calabi-Yau manifolds, the
flow (\ref{krf1}) converges to a Ricci flat metric, if $\Omega$ is a
Calabi-Yau volume form, with the initial metric also Ricci flat in a
different cohomology class \cite{Z3}.

The existence and convergence of the solution are proved by Zhang \cite{Z3} for
the above flow in the case when $[\omega_\infty]$ is K\"{a}hler,
which corresponds to the case considered by Cao in the
classical K\"{a}hler-Ricci flow \cite{C}. When $[\omega_\infty]$ is big,
(\ref{krf1}) may produce singularities at finite time $T<+\infty$. In this
case, the local $C^\infty$ convergence of the flow away from the
stable base locus of $[\omega_T]$ was obtained under the further assumption
that $[\omega_T]$ is semi-ample. When $[\omega_\infty]$ is semi-ample
and big, he obtained the long time existence of the solution and
important estimates and conjectured the convergence even in
this more general setting. In this note, we give a proof to this
conjecture. We will give some definitions before stating our
theorem.

\begin{definition}$[\gamma] \in H^{1,1}(X, \mathbb{C})$ is semi-positive if there exists $\omega \in [\gamma]$ such that $\omega \geq 0$, and is big if $[\gamma]^n:=\int_X \gamma^n >0$.
\end{definition}

\begin{definition} A closed, positive $(1, 1)$- current $\omega$ is called a singular Calabi-Yau metric on $X$ if $\omega$ is a smooth K\"{a}hler metric away from an analytic subvariety $E \subset X$ and satisfies $Ric(\omega)=0$ away from $E$.
\end{definition}

\begin{definition} A volume form $\Omega$ is called Calabi-Yau volume form if $$Ric(\Omega)=-\ddbar \log \Omega=0.$$
\end{definition}

\begin{theorem}\label{Main Theorem} Let $X$ be a K\"ahler manifold
with a K\"ahler metric $\omega_0$. Suppose that $[\omega_\infty] \in H^{1, 1}(X, \mathbb{C})\bigcap
H^2(X, \mathbb{Z})$ is semi-positive and big. Then along the
modified K\"{a}hler-Ricci flow (\ref{krf1}), $\widetilde\omega_t$
converges weakly in the sense of current and converges locally in
$C^\infty$ norm away from the stable base locus of $[\omega_\infty]$
to the unique solution of the degenerate Monge-Amp\`{e}re equation
(\ref{MA}).
\end{theorem}

\begin{corollary} When $X$ is a Calabi-Yau manifold and $\Omega$
is a Calabi-Yau volume form, $\widetilde\omega_t$ converges to a
singular Calabi-Yau metric.
\end{corollary}

\begin{remark} The singular Calabi-Yau metrics are already obtained in \cite{EGZ} on normal Calabi-Yau K\"{a}hler spaces, and obtained by Song-Tian \cite{ST2} and Tosatti \cite{To} independently in the degenerate class on the algebraic Calabi-Yau manifolds. The uniqueness of the solution to the equation (\ref{MA}) in the degenerate case when $[\omega_\infty]$ is semi-positive and big has been studied in \cite{EGZ} \cite{Z2} \cite{DZ} etc. In particular, a stability theorem is proved in \cite{DZ} which immediately implies the uniqueness. In \cite{To}, Tosatti studied the deformation for a family of Ricci flat K\"{a}hler metrics, whose cohomology classes
are approaching a big and nef class. So our deformations give different paths connecting non-singular and singular Calabi-Yau metrics.
\end{remark}

\begin{remark}\label{Kodaira}
As $[\omega_\infty] \in H^{1, 1}(X, \mathbb{C}) \bigcap H^2(X, \mathbb{Z})$, there
exists a line bundle $L$ over $X$ such that $\omega_\infty \in
c_1(L)$. Moreover, $L$ is big when $[\omega_\infty]$ is semi-positive and big \cite{De}.
Hence $X$ is Moishezon. Furthermore, $X$ is algebraic,
for it is K\"{a}hler. Therefore, by applying Kodaira lemma, we see
that for any small positive number $\epsilon \in \mathbb{Q}$, there exists an effective divisor $E$ on $X$, such that
$[L]-\epsilon [E]>0$.
\end{remark}

The structure of the paper is as follows: in the second section, we
define an energy functional whose derivative in $t$ along the flow
is essentially bounded by the $L^2$ norm of the gradient of the
Ricci potential, after deriving uniform estimates for the metric
potential. From this property, we derive the convergence of the
Ricci potential, and furthermore obtain the convergence of the flow.
In the third section, we sketch a proof to the exponential
convergence when $[\omega_\infty]$ is K\"{a}hler also by using the
energy functional defined in the second section.

\bigskip

{\bf{Acknowledgements:}} The author would like to thank Professor Xiaojun Huang and Jian Song for the constant support and encouragement. He would like to thank Professor D. H. Phong for the interest and valuable suggestion. He is also grateful to V. Tosatti and Z.
Zhang for the useful comments.

\section{Proof of Theorem \ref{Main Theorem}}
In this section, we give a proof of Theorem \ref{Main Theorem}. We first define some
notations for the convenience of our later discussions. Let
$\widetilde{\Delta}$ and $\ddt{} -
\widetilde{\Delta}$ be the Laplacian and the heat operator with
respect to the metric ${\widetilde\omega}_t$ and let $\nabla$ denote the gradient operator with respect to metric $\tilde\omega_t$. Let $S$ be the defining section of $E$. By Kodaira Lemma as stated in Remark \ref{Kodaira}.2, there exists a hermitian
metric $h_E$ on $E$ such that $\omega_\infty-\epsilon Ric(h_E)$, denoted
by $\omega_E$, is strictly positive for any $\epsilon$ small. Let $V_t=[\tilde \omega_t]^n$ with $V_t$ uniformly bounded and $V_0=1$.
Write $\dot \varphi = \ddt \varphi$ for simplicity.

\bigskip

Before proving the convergence, we would like to sketch the
uniform estimates of $\varphi$.

First of all, by the standard computation as in \cite{Z3}, the uniform upper bound of $\ddt{} \varphi$ is deduced from the maximum principle. Secondly, by the result of \cite{EGZ} \cite{Z2}, generalizing the theorem of Ko{\l}odziej \cite{K} to the degenerate case,
we have the $C^0$-estimate $\|u\|_{C^0(X)} \leq C$ independent of $t$ where
$u=\varphi-\int_X \varphi \Omega$. Then to estimate $\ddt{} \varphi$
locally, we will calculate $(\ddt{} -
\widetilde{\Delta})[\dot\varphi +
A(u-\epsilon \log\|S\|^2_{h_E})]$, and then the maximum principle yields $\ddt{} \varphi \geq
-C+\alpha \log\|S\|^2_{h_E}$ for $C, \alpha>0$.







Next, we follow the standard second order estimate as in \cite{Y1}
\cite{C} \cite{Si} \cite{Ts}.

Calculating $(\ddt{} -
\widetilde{\Delta})[\log
tr_{\omega_E+e^{-t}\chi} {\widetilde\omega}_t - A(u - \epsilon \log
\|S\|^2_{h_E})]$ and applying maximum principle, we have: $\mid \Delta_{\omega_E}\varphi \mid \leq C$. Then by using the Schauder estimates and third order estimate as in \cite{Z3}, we can obtain
the local uniform estimate: For any $k \geq 0, K \subset\subset
X\setminus E$, there exists $C_{k, K}>0$, such that:

\begin{equation}\label{uniform estimate}
\|u\|_{C^k([0, +\infty) \times K)} \leq C_{k, K}.
\end{equation}

\bigskip

In \cite{Z3},  Zhang proved the following theorem by comparing (\ref{maf1}) with the K\"{a}hler-Ricci flow (\ref{maf2}). We include the detail of the proof here for the sake of completeness. This uniform lower bound appears to be crucial in the proof of convergence.

\begin{theorem}\label{lower bound}
(\cite{Z3})
There exists $C>0$ such that $\ddt{} \varphi \geq -C$ holds
uniformly along (\ref{maf1}) or (\ref{krf1}).
\end{theorem}

We derive a calculus lemma now for later application.

\begin{lemma}\label{lemma}
Let $f(t) \in C^1([0, +\infty))$ be a non-negative function. If
$\int_{0}^{+\infty}f(t)d t<+\infty$ and $\ddt f$ is uniformly
bounded, then $f(t) \rightarrow 0$ as $t \rightarrow +\infty.$
\end{lemma}

\begin{proof} We prove this calculus lemma by contradiction. Suppose there exist a sequence $t_i\rightarrow +\infty$ and $\delta>0$, such that $f(t_i)>\delta$. Since $\ddt f$ is uniformly bounded, there exist a sequence of connected, non-overlapping intervals $I_i$ containing $t_i$ with fixed length $l$, such that $f(t) \geq \frac{\delta}{2}$ over $I_i$. Then $\int_{\bigcup_i I_i}f(t) d t \geq \sum_i l\frac{\delta}{2} \rightarrow +\infty$, contradicting with $\int_{0}^{+\infty}f(t)d t<+\infty$.
\qed\end{proof}

Let ${\widehat\omega}_t=\omega_t+\ddbar \phi$ and consider the
Monge-Amp\`{e}re flow, as well as its corresponding evolution of metrics:

\begin{equation}\label{maf2}
\ddt{}\phi = \log \frac{(\omega_t + \ddbar \phi)^n }{\Omega} -
\phi ;~~~~~~~~~~ \phi(0, \cdot)= 0.
\end{equation}

\begin{equation}\label{krf2}
\ddt{} {\widehat\omega}_t = - Ric( {\widehat\omega}_t ) +
Ric(\Omega) - {\widehat\omega}_t + \omega_\infty;~~~~~~~~~~
 {\widehat\omega}_t(0, \cdot)= \omega_0.
\end{equation}

The following theorem is proved in \cite{Z1} and we include the proof here.

\begin{theorem}There exists $C>0$, such that $\ddt{} \phi
\geq -C$ uniformly along (\ref{maf2}).
\end{theorem}

\begin{proof} Standard computation shows that:





$$(\ddt{} -
\widehat{\Delta}) (\ddt \phi+\frac{\partial
^2 \phi}{\partial t ^2}) \leq -(\ddt \phi+\frac{\partial ^2
\phi}{\partial t ^2}).$$

Then the maximum principle yields:

$$\ddt{} (\phi+\ddt \phi) \leq C e^{-t} ~~and~~ \ddt{} \phi \leq C e^{-\frac{t}{2}}
,$$

which means that $\phi$ and $\phi+\ddt \phi$ are essentially
decreasing along the flow, for example: $\ddt{}(\phi+C e^{-t}) \leq
0.$ As we have the similar estimate that $\phi$ is locally uniformly
bounded away from $E$, then $\phi$ converges to some
$\omega_\infty-$ plurisubharmonic function away from $E$, which also
yields that the pointwise limit of $\ddt\phi$ is $0$ away from $E$
by Lemma \ref{lemma}. On the other hand, the weak convergence
of (\ref{maf2}) is obtained in \cite{Ts} \cite{TZha}. Suppose
$\phi_\infty$ is the weak solution to the degenerate
Monge-Amp\`{e}re equation as the limit equation with $\phi+\dot\phi
\rightarrow \phi_\infty$ away from $E$ when $t \rightarrow +\infty$:
$$e^{\phi_\infty}=(\omega_\infty+\ddbar \phi_\infty)^n~~ \leftarrow ~~e^{\phi+\dot\phi}=(\omega_t+\ddbar \phi)^n.$$
Furthermore, from the pluri-potential theory \cite{EGZ} \cite{Z2}
and equation (\ref{maf2}), we know that $-C \leq \phi, \phi_\infty
\leq C$ uniformly. With the essentially decreasing property, we know
that away from $E$:
$$\phi+\dot\phi+C e^{-t} \geq \phi_\infty \geq -C,$$ yielding
$\dot\phi \geq -C$ away from $E$. The theorem is proved with
$\dot\phi$ smooth on $X$.

\qed\end{proof}

We are now ready to give Zhang's proof to Theorem \ref{lower bound}.

\bigskip

{\bf{Proof of Theorem \ref{lower bound}:}}

Fix $T_0>0$. Let $\kappa(t,
\cdot)=(1-e^{-T_0})\dot\varphi+u-\phi(t+T_0)$ with $\kappa(0,
\cdot) \geq -C_1$. Then:

\begin{eqnarray*}
\ddt{} \kappa(t, \cdot) &=&
\widetilde{\Delta}((1-e^{-T_0})\dot\varphi + u) + \dot u -
\dot\phi(t+T_0) - n +
tr_{{\widetilde\omega}_t} \omega_{t+T_0}\\
&=& \widetilde{\Delta}((1-e^{-T_0})\dot\varphi + u - \phi(t+T_0))
+ \dot u - \dot\phi(t+T_0) -n + tr_{{\widetilde\omega}_t} \widehat\omega_{t+T_0}\\
&\geq& \widetilde{\Delta}((1-e^{-T_0})\dot\varphi + u -
\phi(t+T_0)) + \dot\varphi -C_2 + n(\frac{C_3}{e^{\dot\varphi}
})^\frac{1}{n},
\end{eqnarray*}

where the fact: $\dot u \geq \dot\varphi-C, -C \leq
\ddt{} \phi \leq C$ and $\widehat\omega_t \geq C_3\Omega$ are used.

Suppose $\kappa(t, \cdot)$ achieves minimum at $(t_0, p_0)$ with
$t_0>0$. Then the maximum principle yields $\dot\varphi(t_0, p_0)
\geq -C_4$. Hence $\dot\varphi$ is bounded from below. Suppose
$\kappa(t, \cdot)$ achieves minimum at $t=0$. Then the theorem
follows trivially.
\qed

\bigskip

Inspired from the Mabuchi $K$-energy in the study on the
convergence of K\"{a}hler-Ricci flow on the Fano manifolds, we
similarly define an energy functional as follows:
$$\nu(\varphi)=\int_X \log \frac{(\omega_t + \ddbar \varphi)^n
}{\Omega} (\omega_t + \ddbar \varphi)^n.$$

Next, we will give some properties of $\nu(\varphi)$ and then a key lemma for the proof of the theorem.

\begin{proposition}$\nu(\varphi)$ is well-defined and there exists $C>0$, such that $-C < \nu(\varphi) < C$ along (\ref{krf1}).
\end{proposition}

\begin{proof}It is easy to see that $\nu(\varphi)$ is well-defined. If we rewrite
$$\nu(\varphi)=\int_X \dot\varphi \tilde\omega_t^n,$$
then $\nu(\varphi)$ is uniformly bounded from above and below by the upper and lower bound of $\dot\varphi$. Here, we derive the uniform lower bound by Jensen's inequality without using Theorem \ref{lower bound}:

$$\nu(\varphi)=- V_t \int_X \log \frac{\Omega}{\tilde\omega_t ^n} \frac{\tilde\omega_t ^n}{V_t} \geq - V_t \log \int_X \frac{\Omega}{V_t} \geq -C,$$
as $V_t$ is uniformly bounded.
\qed\end{proof}

\begin{proposition}
There exists constant $C>0$ such that for all $t>0$ along the flow (\ref{krf1}):
\begin{equation}\label{energy2}
\ddt{} \nu(\varphi) \leq -\int_X \|\nabla
\dot\varphi\|^2_{\widetilde\omega_t} \widetilde\omega_t ^n + C
e^{-t}.
\end{equation}
\end{proposition}

\begin{proof} Along the flow (\ref{krf1}), we have:
\begin{eqnarray*}\label{energy1}
\ddt{} \nu(\varphi) &=& \int_X \tilde\Delta\dot\varphi \widetilde\omega_t ^n -
e^{-t}\int_X tr_{\widetilde\omega_t}\chi \widetilde\omega_t ^n - n e^{-t}\int_X \dot\varphi \chi \wedge \widetilde\omega_t ^{n-1} + n \int_X \dot\varphi \ddbar\dot\varphi \wedge \widetilde\omega_t ^n\\
&=& -\int_X \|\nabla
\dot\varphi\|^2_{\widetilde\omega_t} \widetilde\omega_t ^n -
ne^{-t}\int_X \chi \wedge \widetilde\omega_t ^{n-1} - ne^{-t}
\int_X \dot\varphi \chi \wedge \widetilde\omega_t ^{n-1}\\
&=& -\int_X \|\nabla
\dot\varphi\|^2_{\widetilde\omega_t} \widetilde\omega_t ^n -
n [\chi][\omega_t] ^{n-1} e^{-t} - ne^{-t}
\int_X \dot\varphi (\omega_0-\omega_\infty) \wedge
\widetilde\omega_t
^{n-1}\\
&\leq& -\int_X \|\nabla \dot\varphi\|^2_{\widetilde\omega_t}
\widetilde\omega_t ^n + n [\chi][\omega_t] ^{n-1} e^{-t} + Ce^{-t}\int_X (\omega_0 +
\omega_\infty) \wedge \omega_t
^{n-1}\\
&=& -\int_X \|\nabla \dot\varphi\|^2_{\widetilde\omega_t}
\widetilde\omega_t ^n + [n \chi + C \omega_0+ C \omega_{\infty}][\omega_t]^{n-1} e^{-t}\\
&\leq& -\int_X \|\nabla \dot\varphi\|^2_{\widetilde\omega_t}
\widetilde\omega_t ^n + C'e^{-t}.
\end{eqnarray*}

Notice that we used the evolution of $\dot\varphi$ and integration by parts in the first two equalities and the uniform bound of $\dot\varphi$ in the first inequality and the last inequality holds since $\omega_t$ is uniformly bounded.
\qed\end{proof}

\begin{lemma}\label{gradient}
On each $K\subset\subset X\setminus E$, $\|\nabla \dot{\varphi}(t)\|^2_{\widetilde\omega_t} \rightarrow 0$ as $t\rightarrow +\infty$ uniformly.

\end{lemma}

\begin{proof} Integrating (\ref{energy2}) from $0$ to $T$, we have:

$$-C \leq \nu(\varphi)(T)-\nu(\varphi)(0) \leq -\int_0^T \int_X
\|\nabla \dot\varphi\|^2_{\widetilde\omega_t} \widetilde\omega_t
^n dt + C$$ for some constant $C>0$. It follows that

$$\int_0^{+\infty} \int_X \|\nabla
\dot\varphi\|^2_{\widetilde\omega_t} \widetilde\omega_t ^n dt \leq
2C,$$ by letting $T \rightarrow +\infty$. Hence, (\ref{uniform
estimate}) and Lemma \ref{lemma} imply that for any compact set
$K' \subset X \setminus E$,

\begin{equation}\label{limit}
\int_{K'} \| \nabla \dot{\varphi}(t, \cdot)
\|_{\widetilde\omega_{t}} ^2 \tilde{\omega}_{t}^n \rightarrow 0.
\end{equation}

Now, assume that there exists $\delta>0$, $z_j \in K $ and $t_j
\rightarrow \infty$ such that $\| \nabla \dot{\varphi}(t_j, z_j)
\|_{\widetilde\omega_{t_j}} ^2 > \delta.$ It follows from
(\ref{uniform estimate}) that $\| \nabla \dot{\varphi}(t_j, z)
\|_{\widetilde\omega_{t_j}} ^2 > \frac{\delta}{2}$ for  $z\in
B(z_j, r) \subset K'$ for $K \subset K' \subset \subset X\setminus
E$ and some $r>0$. This  contradicts with (\ref{limit}).

\qed\end{proof}

{\bf{Proof of Theorem \ref{Main Theorem}:}}

First of all, we want to show that for any $k \in \mathbb{Z}, K
\subset\subset X \setminus E$, $u(t) \rightarrow \psi$ in
$C^{\infty}(K)$.

Exhaust $X \setminus E$ by compact sets $K_i$ with $K_i \subset
K_{i+1}$ and $\bigcup_i K_i = X\setminus E$. As $\| u
\|_{C^k(K_i)} \leq C_{k, i}$, after passing to a subsequence $t_{ij}$, we
know $u(t_{ij}) \rightarrow \psi$ in $C^{\infty}(K_i)$ topology.
By picking up the diagonal subsequence of $u(t_{ij})$, we know
that $\psi \in L^{\infty}(X) \bigcap C^{\infty}(X \setminus E)
\bigcap PSH(X \setminus E, \omega_\infty)$. Furthermore, $\psi$
can be extended over $E$ as a function in $PSH(X, \omega_\infty)$.
Taking gradient of (\ref{maf1}), by Lemma \ref{gradient}, we have
on $K_i$ as $t_{ij} \rightarrow +\infty$,

$$\nabla \dot\varphi = \nabla \log \frac{(\omega_t +
\ddbar \varphi)^n }{\Omega} \rightarrow 0 = \nabla \log
\frac{(\omega_\infty + \ddbar \psi)^n }{\Omega}.$$ Hence, we know
that $\log \frac{(\omega_\infty + \ddbar \psi)^n }{\Omega} =
constant$ on $X \setminus E$ . Then the constant can only be $0$
as $\psi$ is a bounded pluri-subharmonic function and $\int_X
\omega_{\infty}^n=\int_X \Omega$, which means that $\psi$ solves
the degenerate Monge-Amper\`{e} equation (\ref{MA}) globally in the sense of current and strongly on $X \setminus E$. Furthermore, we notice that $\int_X \psi \Omega=0$ as $\psi$ is bounded.

Suppose $u(t) \nrightarrow \psi$ in $C^{\infty}(K)$ for some
compact set $K \subset X \setminus E$, which means that there
exist $\delta>0$, $l \geq 0$, $K' \subset\subset X\setminus E$,
and a subsequence $u(s_j)$ such that $\|u(s_j) -
\psi\|_{C^{l}(K')}> \delta$. While $u(s_j)$ are bounded in
$C^{k}(K)$ for any compact set $K$ and $k \geq 0$,  from the above
argument, we know that by passing to a subsequence, $u(s_j)$
converges to $\psi'$ in $C^{\infty}(K)$ for any $K \subset \subset
X\setminus E$, where $\psi'$ is also a solution to equation
(\ref{MA}) under the normalization $\int_X \psi' \Omega=0$, which
has to be $\psi$ by the uniqueness of the solution to (\ref{MA}).
This is a contradiction. It thus follows that $u(t) \rightarrow
\psi$ in $L^p(X)$ for any $p>0$.

Notice that  $u(t), \psi$ are uniformly bounded.  Integrating by
part, we easily  deduce that $\widetilde\omega_t=\omega_t+\ddbar u
\rightarrow \omega_\infty+\ddbar \psi$ weakly in the sense of
current. The proof of the theorem is complete.

\qed





\section{Remarks on non-degenerate case}
In the case when $[\omega_\infty]$ is K\"{a}hler, the convergence of
(\ref{krf1}) has already been proven by Zhang in \cite{Z3} by
modifying Cao's argument in \cite{C}. However, by using the
functional $\nu(\varphi)$ defined in the previous section, we will
have an alternating proof to the convergence, without using
Li-Yau's Harnack inequality. We will sketch the proof in this section. We believe that this point
of view is well-known to the experts.

Firstly, under the same normalization $u= \varphi-\int_X \varphi
\Omega$, we will have the following uniform
estimates (\cite{Z3}): for any integer $k \geq 0$, there exists $C_k >0$, such
that

$$\|u\|_{C^k([0, +\infty)\times X)} \leq C_k .$$

Secondly, following the convergence argument as in the previous
section, we will obtain the $C^\infty$ convergence of $\tilde\omega_t$ along
(\ref{maf1}). More precisely, $u \rightarrow \psi$ in
$C^\infty$-norm with $\psi$ solving (\ref{MA}) as a strong
solution. In particular, $\dot\varphi \rightarrow 0$ in $C^\infty$-norm as $t \rightarrow +\infty$. Let $\tilde\omega_\infty=\omega_\infty+\ddbar \psi$ be
the limit metric. Furthermore, we have the bounded geometry along
(\ref{krf1}) for $0\leq t \leq +\infty$:
\begin{equation}\label{bounded geometry}
\frac{1}{C} \tilde\omega_{\infty} \leq \tilde\omega_t \leq C \tilde\omega_{\infty}.
\end{equation}

Finally, we need to prove the exponential convergence: $\|\tilde\omega_t - \tilde\omega_{\infty}\|_{C^k(X)} \leq C_k e^{-\alpha t}$ and $\|u(t)-\psi\|_{C^k(X)} \leq C_k e^{-\alpha t}$, for some $C_k, \alpha>0$. Then it is sufficient to prove: for any integer $k \geq 0$,
there exists $c_k>0$, such that

$$\| D^k \dot\varphi \|^2_{\omega_0} \leq c_k e^{-\alpha t}.$$

Essentially, by following the proof of Proposition 10.2 in the case of holomorphic vector fields $\eta(X)=0$ in \cite{CT}, we can also prove the following proposition.

\begin{proposition}\label{exponential estimate}
Let $c(t)=\int_X \ddt\varphi \tilde{\omega}_{t}^n$. There exists $\alpha>0$ and $c'_k>0$ for any integer $k \geq 0$,
such that
$$\int_X \| D^k (\ddt \varphi-c(t)) \|^2_{\widetilde\omega_t} \tilde{\omega}_{t}^n \leq c'_k e^{-\alpha t}.$$
\end{proposition}

By following the argument of Corollary 10.3 in \cite{CT}, we can prove that
$|c(t)| \leq C e^{-\alpha t}$. On the other hand, since
the geometry is bounded along the flow (\ref{bounded geometry}), the Sobolev
constants are uniformly bounded. By using the Sobolev inequality,
we have:

$$\|D^k (\dot\varphi-c(t))\|^2_{\omega_0} \leq c_k e^{-\alpha t}.$$

The exponential convergence is thus obtained combining the estimate
of $c(t)$.

Yuan Yuan (yuanyuan$@$math.rutgers.edu), Department of
Mathematics, Rutgers University - Hill Center for the Mathematical
Sciences, 110 Frelinghuysen Road, Piscataway, NJ 08854-8019.

\bigskip


\small
\begin{thebibliography}{99}
\bibitem[C]{C} Cao, H.-D. {\em Deformation of K\"{a}hler metrics
to K\"{a}hler Einstein metrics on compact K\"{a}hler manifolds},
Invent. Math. 81 (1985), no. 2, 359--372,
\bibitem[CT]{CT} Chen, X. X. and Tian, G. {\em Ricci flow
on K\"{a}hler-Einstein surfaces}, Invent. Math. 147 (2002), no. 3,
487--544,
\bibitem[De]{De} Demailly, J.-P. {Complex analytic and algebraic geometry},
\bibitem[DZ]{DZ} Dinew, S. and Zhang, Z. {\em Stability of Bounded Solutions for Degenerate
Complex Monge-Amp\`{e}re Equations}, Preprint,
\bibitem[Do]{Do} Donaldson, S. {\em Scalar curvature and projective embeddings I}, J. Differential Geom.  59  (2001),  no. 3, 479--522,
\bibitem[EGZ]{EGZ} Eyssidieux, P.; Guedj, V. and Zeriahi,
{\em A Singular K\"{a}hler-Einstein metrics}, to appear in the
Journal of A. M. S.,
\bibitem[H]{H} Hamilton, R. {\em Three-manifolds with positive Ricci curvature}, J. Differential Geom.  17  (1982), no. 2, 255--306,
\bibitem[K]{K} Ko{\l}odziej, S. {The complex Monge-Amp\`{e}re equation and pluripotential theory},  Mem. Amer. Math. Soc.  178  (2005),  no. 840, x+64 pp,
\bibitem[Pe]{Pe} Perelman, G. {unpublished notes},
\bibitem[PS]{PS} Phong, D. H. and Sturm, J. {\em On stability and the convergence of the K\"{a}hler-Ricci flow}, J. Differential Geom.  72  (2006),  no. 1, 149--168,
\bibitem[PSSW]{PSSW} Phong, D. H.; Song, J.; Sturm, J. and Weinkove, B. {\em The K\"{a}hler-Ricci flow and the $\bar\partial$ operator on vector fields},  J. Differential Geom.  81  (2009),  no. 3, 631--647,
\bibitem[Si]{Si} Siu, Y. T. {Lectures on Hermitian-Einstein metrics for stable bundles and K\"{a}hler-Einstein metrics}, DMV Seminar, 8. Birkh\"{a}user Verlag, Basel, 1987. 171 pp,
\bibitem[ST1]{ST1} Song, J. and Tian, G. {\em The K\"{a}hler-Ricci
flow on surfaces of positive Kodaira dimension}, Invent. Math. 170
(2007), no. 3, 609--653,
\bibitem[ST2]{ST2} Song, J. and Tian, G. {\em Canonical measures and Kahler-Ricci
flow (I)}, Preprint,
\bibitem[ST3]{ST3} Song, J. and Tian, G. {\em Canonical measures and Kahler-Ricci
flow (II)},
\bibitem[SY]{SY} Song, J. and Yuan, Y. {\em K\"{a}hler-Ricci flow on Calabi-Yau
manifolds},
\bibitem[T1]{T1} Tian, G. {\em K\"{a}hler-Einstein metrics with positive scalar curvature}, Invent. Math.  130  (1997),  no. 1, 1--37,
\bibitem[T2]{T2} Tian, G. {Canonical metrics in K\"ahler geometry},
Lectures in Mathematics, ETH Z\"urich, Birkhauser Verlag, Basel
2000,
\bibitem[TZha]{TZha} Tian, G. and Zhang, Z. {\em On the
K\"{a}hler-Ricci flow on projective manifolds of general type},
Chinese Ann. Math. Ser. B 27 (2006), no. 2, 179--192,
\bibitem[TZhu]{TZhu} Tian, G. and Zhu, X. {\em Convergence of K\"{a}hler-Ricci flow}, J. Amer. Math. Soc.  20  (2007), no. 3, 675--699,
\bibitem[To]{To} Tosatti, V. {\em  Limits of Calabi-Yau metrics when the Kahler class
degenerates}, to appear in the Journal of the European
Mathematical Society,
\bibitem[Ts]{Ts} Tsuji, H. {\em Existence and degeneration of K\"ahler-Einstein
metrics on minimal algebraic varieties of general type}, Math.
Ann. 281 (1988), 123-133,
\bibitem[We]{We} Weinkove, B. {\em The Calabi-Yau equation on almost-K\"{a}hler four-manifolds}, J. Differential Geom. 76 (2007), no. 2, 317--349,
\bibitem[Y1]{Y1} Yau, S.-T. {\em On the Ricci curvature of a compact K\"ahler
manifold and the complex Monge-Amp\`ere equation, I}, Comm. Pure
Appl. Math. 31 (1978), 339-411,
\bibitem[Y2]{Y2} Yau, S.-T. {\em Problem section} Seminar on Differential Geometry,  pp. 669--706, Ann. of Math. Stud., 102, Princeton Univ. Press, Princeton, N.J., 1982,
\bibitem[Z1]{Z1} Zhang, Z. {\em Degenerate Monge-Amp\`{e}re equations over projective
manifolds}, Thesis,
\bibitem[Z2]{Z2} Zhang, Z. {\em On degenerate Monge-Amp\`{e}re equations over closed K\"{a}hler
manifolds}, Int. Math. Res. Not. 2006, Art. ID 63640, 18 pp,
\bibitem[Z3]{Z3} Zhang, Z. {\em A Modified K\"{a}hler-Ricci Flow},
to appear in Math. Ann.
\end{thebibliography}
\end{document}